\documentclass[11pt]{amsart}
\usepackage{amssymb}
\usepackage{amsmath}
\newtheorem{theorem}{Theorem}[section]
\newtheorem{lemma}[theorem]{Lemma}

\newtheorem{definition}[theorem]{Definition}

\newtheorem{proposition}[theorem]{Proposition}

\newtheorem{lem-def}[theorem]{Lemma-Definition}

\def\as#1{\renewcommand\arraystretch{#1}}

\def\a1x{A_1[x]}
\def\bb{\mathcal{B}}

\def\eps{\epsilon}

\def\gl#1#2{\op{GL}_{#1}(#2)}

\def\iso{\ \lower .6ex\hbox{$\stackrel{\lra}{\mbox{\tiny $\sim\,$}}$}\ }

\def\lra{\longrightarrow}
\def\m{\mathfrak{m}}

\def\mn{\operatorname{Min}}
\def\mx{\operatorname{Max}}

\def\op{\operatorname}

\def\p{\mathfrak{p}}

\def\q{\mathfrak q}
\def\Q{\mathbb Q}

\def\rd{\operatorname{red}}
\def\rr{\mathcal{R}}
\def\t{\theta}

\def\ut{\op{UT}_n(A)}
\def\Z{\mathbb Z}

\newtheorem{alg}[theorem]{Algorithm}
\newlength{\alginputwidth}
\newlength{\algtmp}
\title[Reduced normal form of local integral bases]{Reduced normal form of local integral bases}

\author{Nath\'alia Moraes de Oliveira}
\address{Departament de Matem\`{a}tiques,
         Universitat Aut\`{o}noma de Barcelona,
         Edifici C\\ E-08193 Bellaterra, Barcelona, Catalonia, Spain}
\email{noliveira@mat.uab.cat}

\author{Enric Nart}

\address{Departament de Matem\`{a}tiques,
         Universitat Aut\`{o}noma de Barcelona,
         Edifici C\\ E-08193 Bellaterra, Barcelona, Catalonia, Spain}
\email{nart@mat.uab.cat}


\thanks{Partially supported by CNPq 204224$\mid$2104-4 from the Conselho Nacional de Desenvolvimento Cient\'\i fico e Tecnol\'ogico, and MTM2013-40680 from the Spanish MEC}

\makeatletter
\@namedef{subjclassname@2010}{%
\textup{2010} Mathematics Subject Classification}


\begin{document}
\begin{abstract}
We introduce a canonical form for reduced bases of integral closures of discrete valuation rings, and we describe an algorithm for computing a basis in reduced normal form.
This normal form has the same applications as the Hermite normal form: identification of isomorphic objects, construction of global bases by patching local ones, etc.
but in addition the bases are reduced, which is a crutial property for several important applications. Except for very particular cases, a basis in Hermite normal form cannot be reduced.   
\end{abstract}
\maketitle

\section{Introduction}
Let $k$ be a field and $x$ an indeterminate.
The first approach to a theory of lattices over the polynomial ring $k[x]$ goes back to Mahler \cite{mahler}. In \cite[\S16]{HLenstra}, H. Lenstra gave a brief sketch of the essential features of the theory, which has been developed in full scope by Bauch \cite{bauch2}. 

The role of the norm determined by a quadratic positive definite form, in the  classical theory of lattices over $\Z$, is undertaken by a certain \emph{length function} $d$ defined over a finite dimensional vector space over $k(x)$. For the vector space underlying a finite field extension $L/k(x)$, we can consider:
$$
d\colon L^*\longrightarrow \Q,\quad d(\alpha)=-\mn\{w_i(\alpha)\mid 1\le i\le t\}, 
$$
where $w_1,\dots,w_t$ are the valuations on $L$ extending the valuation $v_\infty$ on $k(x)$, characterized by $v_\infty(a)=-\deg(a)$ for any polynomial $a\in k[x]$. In this way, $d$ is a kind of extension of the degree function on $k[x]$.

A relevant concept is that of \emph{reduced basis} of a lattice with respect to the given length function. W. M. Schmidt used reduced bases of integral closures of certain subrings of function fields of curves over  finite fields, as a crutial tool for the design of algorithms to compute bases of the Riemann-Roch spaces  attached to divisors of the curve \cite{schmidt,schoernig,hess,bauch2}.

In this paper, we study reduced bases of integral closures of arbitrary discrete valuation rings. 

Let $A$ be a discrete valuation ring with field of fractions $K$. 
Let $L/K$ be a finite field extension, and $B$ the integral closure of $A$ in $L$, which we suppose to be finitely generated as an $A$-module. Let $v$ be the valuation on $A$ and $w_1,\dots,w_t$ the valuations on $L$ extending $v$.
The notion of reduced families of elements in $L$ with respect to the function 
$$
w\colon L^*\longrightarrow \Q,\quad w(\alpha)=\mn\{w_i(\alpha)\mid 1\le i\le t\}
$$ was already introduced in \cite{bases} as a tool to prove that certain families of integral elements constitute an $A$-basis of $B$.   

In section \ref{secRed}, we develop in a more comprehensive way the properties of reduced families in this general context. In Theorem \ref{distribution} we compute the multiset of $w$-values of a reduced integral basis, which turns out to be independent of the basis. Also, in Theorem \ref{transition} we find the structure of the transition matrices between reduced integral bases.

In section \ref{secTriang}, we present a triangulation routine to convert a given reduced integral basis into a triangular one, without destroying reduceness. This has many practical applications. For any task involving the previous computation of a reduced integral basis (like the computation of Riemann Roch spaces of function fields)  the computational cost is diminished if we use a triangular reduced integral basis. Specially, triangular integral bases facilitate the computation of global integral bases by patching local ones, with the aid of the chinese remainder theorem.

In section \ref{secRNF} we introduce a normal form for triangular reduced integral bases. 
Finally, in section \ref{secComp} we discuss some computational issues concerning the computation or integral bases in reduced normal form, and we exhibit a concrete example.

\section{Reduced integral bases}\label{secRed}
Let $v\colon K\to \Z\cup\{\infty\}$ be a discrete valuation on a field $K$. Let $A$ be the valuation ring, $\pi\in A$ an uniformizer, $\m=\pi A$ the maximal ideal of $A$, and $k=A/\m$ the residue class field. 

Let $L/K$ be a monogene finite field extension of $K$ of degree $n>1$; that is, $L=K(\t)$ for a certain $\t\in L$ which is the root of some monic irreducible polynomial $f\in A[x]$ of degree $n$.

Let $B\subset L$ be the integral closure of $A$ in $L$. 
The ring $B$ is a Dedekind domain, which we assume to be finitely generated as an $A$-module. This is the case, for instance, when $L/K$ is separable, or $K$ is complete, or $A$ is a finitely generated algebra over a field \cite[Ch.I, \S4]{serre}. 

Under this assumption, $B$ is a free $A$-module of rank $n$. 
An $A$-basis of $B$ is called an \emph{integral basis} of $L/K$.

Although integral bases are ordered families of elements in $B$, sometimes we forget the ordering and consider integral bases merely as subsets of $B$. 

Let $w_1,\dots, w_t$ be the valuations on $L$ extending $v$. For each $w_i$, let $B_i\subset L$ be the valuation ring, $\m_i$ the maximal ideal of $B_i$ and $k_i=B_i/\m_i$ the residue class field. Denote $f_i=[k_i\colon k]$ and $e_i=e(w_i/v)$. The ramification index $e_i$ is characterized by the property $w_i(L^*)=e_i^{-1}\Z$.     
In this situation, it holds the well-known relation $\sum_ie_if_i=n$.

Consider the following quasi-valuation extending $v$ to $L$: 
$$
w\colon L\lra \Q\cup\{\infty\},\quad w(\alpha)=\mn\nolimits\{w_i(\alpha)\mid 1\le i\le t\},
$$
For $\alpha,\beta\in L$, $a\in K$ and $m\in\Z$, this mapping $w$ satisfies:\medskip

(1) \ $w(\alpha\beta)\ge w(\alpha)+w(\beta)$, and equality holds if $\beta=\alpha^m$,\medskip

(2) \ $w(a\beta)=w(a)+w(\beta)=v(a)+w(\beta)$, \medskip

(3) \ $w(\alpha+\beta)\ge \mn\{w(\alpha),w(\beta)\}$, and equality holds if $w(\alpha)\ne w(\beta)$.

\begin{lemma}\label{values}
$w(L^*)=\cup_{i=1}^tw_i(L^*)=\cup_{i=1}^te_i^{-1}\Z$. 
\end{lemma}

\begin{proof}
By the very definition, $w(L^*)\subset\cup_{i=1}^tw_i(L^*)$. Since the valuations $w_1,\dots,w_t$ are pairwise independent, for each $1\le  i\le t$ there exists an element  $\alpha_i\in B$ with $w(\alpha_i)=e_i^{-1}$. Hence, $w_i(L^*)=w(\{\alpha_i^m\mid m\in \Z\})$ is contained in $w(L^*)$ for all $1\le  i\le t$.  
\end{proof}

Since $B=B_1\cap\cdots\cap B_t$, the integral elements are characterized by $$B=\{\alpha\in L\mid w(\alpha)\ge0\}.$$  
Also, the subset $\bb\subset B$ formed by an integral basis satisfies $w(\bb)\subset [0,1)$. In fact, if $\alpha\in\bb$ has $w(\alpha)>1$ then $\alpha/\pi$ is integral and it does not belong to the $A$-module generated by $\bb$.

\begin{definition}
A subset $\{\alpha_1,\dots,\alpha_d\}\subset L^*$ is called \emph{reduced} if for all $a_1,\dots,a_d\in K$, one has:
\begin{equation}\label{reduceness}
w\left(\sum\nolimits_{1\le j\le d}a_j\alpha_j\right)=\mn\{w(a_j\alpha_j)\mid 1\le j\le d\}.
\end{equation}
\end{definition}

The left and right hand-side of (\ref{reduceness}) increase by $\nu\in\Z$ if we replace each $a_j$ with $a_j\pi^\nu$. Thus, in order to check the equality (\ref{reduceness}) we can assume that all $a_j$ belong to $A$ and not all of them belong to $\m$.   

The following property follows immediately from the definition.

\begin{lemma}\label{scale}
If $\{\alpha_1,\dots,\alpha_d\}$ is reduced, then for all $a_1,\dots,a_d\in K$ the set $\{a_1\alpha_1,\dots,a_d\alpha_d\}$ is reduced.
\end{lemma}

It is easy to check that a reduced set is always $K$-linearly independent. Further, any reduced set $\left\{\alpha_j\mid 1\le j\le n\right\}$ of cardinality $n$ determines a reduced integral basis $\left\{\alpha_j/\pi^{\lfloor w(\alpha_j)\rfloor}\mid 1\le j\le n\right\}$, as the following result shows.

\begin{lemma}\label{reducedbasis}
A reduced set $\bb=\{\alpha_1,\dots,\alpha_n\}\subset L^*$ such that $w(\bb)\subset [0,1)$ is a reduced integral basis of $L/K$.
\end{lemma}

\begin{proof}
The assumption on $w(\bb)$ shows that $\bb\subset B$. Let us prove that $\bb$ generates $B$ as an $A$-module. 

Any $\alpha\in B$ may be expressed as $\alpha=\sum_{j=1}^na_j\alpha_j$, for some $a_1,\dots,a_n\in K$. By reduceness, for all $j$ we have
$$
w(a_j\alpha_j)\ge w(\alpha)\ge0.
$$
Since $w(\alpha_j)<1$ and $w(a_j)$ is an integer, this implies $w(a_j)\ge0$, or equivalently, $a_j\in A$.
\end{proof}

Our aim is to show that all reduced integral bases $\bb$ of $L/K$ have the same multiset $w(\bb)$. We want to compute the cardinality of the subsets:
$$
\bb_\delta=\{\alpha\in \bb\mid w(\alpha)=\delta\}\subset \bb,\quad \delta\in w(L^*).
$$
To this end, we need a certain criterion for reduceness.

For any $\delta\in w(L^*)$, consider the $A$-modules: $$L_\delta=\{\alpha\in L\mid w(\alpha)\ge \delta\}\supset L_\delta^+=\{\alpha\in L\mid w(\alpha)>\delta\}.$$ 

Since $\m L_\delta\subset L_\delta^+$, the quotient $L_\delta/L_\delta^+$ has a structure of $k$-vector space.

\begin{definition}
Consider the $k$-vector space $V=\prod_{i=1}^tk_i$ of dimension $\sum_{i=1}^tf_i$.
For each $1\le i\le t$ let us fix some uniformizer $\pi_i\in\m_i$. 

For all $\delta\in w(L^*)$ we define a reduction map:
$$
\rd_\delta\colon L_\delta\lra V,\quad \rd_\delta(\alpha)=(\alpha_{\delta,i})_{1\le i\le t},\quad
\alpha_{\delta,i}=
\alpha\pi_i^{-\lfloor e_i\delta\rfloor}+\m_i.
$$
\end{definition}

Clearly, $\rd_\delta$ is an homomorphism of $A$-modules and $\ker(\rd_\delta)=L_\delta^+$. Hence, it induces an embedding of $L_\delta/L_\delta^+$ as a $k$-subspace of $V$.

\begin{theorem}\cite{schmidt,schoernig},\cite[Lem. 5.7]{bases}\label{criterion}
Let $\bb\subset L$ with  $w(\bb)\subset [0,1)$. Then, $\bb$ is reduced if and only if $\rd_\delta(\bb_\delta)\subset V$ is a $k$-linearly independent family for all $\delta\in w(\bb)$.
\end{theorem}

\begin{definition}\label{multiset}
Given a set $E$, we indicate by $\{e^{m_e}\mid e\in E\}$ the multiset which contains each element $e\in E$ with multiplicity $m_e$.
\end{definition}

\begin{theorem}\label{distribution}
Let $E=w(L^*)\cap [0,1)$, and for each $\delta\in E$ consider:  
$$I_\delta=\left\{1\le i\le t\mid \delta\in e_i^{-1}\Z\right\},\quad f_\delta=\sum\nolimits_{i\in I_\delta}f_i.$$

Then, for any reduced integral basis $\bb$ we have $\#\bb_\delta=f_\delta$. In other words,
the multiset $w(\bb)$ is equal to $W_{L/K}:=\{\delta^{f_\delta}\mid \delta\in E\}$.  
\end{theorem}

\begin{proof}
By Lemma \ref{values}, $E=\cup_{i=1}^tE_i$, where $E_i$ are the sets:
$$E_i=e_i^{-1}\Z\cap [0,1)=\{0,e_i^{-1},\dots,(e_i-1)e_i^{-1}\},\quad 1\le i\le t.$$
For each $i$ consider the multiset $X_i=\{\delta^{f_i}\mid \delta \in E_i\}$.
Let $X=\coprod_{i=1}^tX_i$ be the formal disjoint union of these multisets. The natural inclusions $X_i\subset W_{L/K}$ induce a bijection of multisets between $X$ and $W_{L/K}$. Hence,
$$
\sum\nolimits_{\delta\in E}f_\delta=\#W_{L/K}=\#X=\sum\nolimits_{i=1}^te_if_i=n.
$$

On the other hand, $\rd_\delta(\bb_\delta)\subset\prod_{i\in I_\delta}k_i$ for all $\delta\in E$. In fact, for $\alpha\in\bb_\delta$ and $j\not\in I_\delta$, we have $w_j(\alpha)\ne\delta=w(\alpha)$; hence $w_j(\alpha)>\delta$ and $\alpha_{\delta,j}=0$.

By Theorem \ref{criterion}, $\#\bb_\delta\le f_\delta$ for all $\delta\in E$. Therefore,
$$n=\sum\nolimits_{\delta\in E}\#\bb_\delta\le \sum\nolimits_{\delta\in E}f_\delta=n,
$$
and the result follows.
\end{proof}

We end this section with a description of the transition matrices between reduced integral bases.\medskip

\noindent{\bf Notation.} {\it For any matrix $T\in A^{m\times m}$ we denote by $\overline{T}\in k^{m\times m}$ the matrix obtained by applying reduction modulo $\m$ to all entries in $T$}. \medskip

\begin{definition}\label{orthonormal}
An \emph{orthonormal basis} of $L/K$ is a reduced integral basis $(\alpha_1,\dots,\alpha_n)\in B^n$ ordered by increasing $w$-values: $w(\alpha_1)\le\cdots\le w(\alpha_n)$.
\end{definition}

This terminology is taken from \cite{bauch2}, where lattices over the polynomial
ring $k[x]$ are studied. In that context, $v=v_\infty$ is the valuation ``at the point of infinity", which restricted to $k[x]$ is equal to $-\deg$. For a finite extension $L/k(x)$ the function $-w$ is interpreted as a length function on $L$, playing a role analogous to the norm determined by a positive definite quadratic form in the theory of lattices over $\Z$.  

Theorem \ref{transition} below is inspired by \cite[Thm. 1.27, Lem. 1.28]{bauch2} too.

\begin{definition}\label{orthogroup}
Let $n=m_1+\cdots +m_\kappa$ be a partition of $n$ into a sum of positive integers. For any $T\in A^{n\times n}$ this partition induces a decomposition of $T$ into a $\kappa\times \kappa$ matrix of blocks:
$$
T=\left(T_{ij}\right),\quad T_{ij}\in A^{m_i\times m_j},\ 1\le i,j\le \kappa.
$$

The \emph{orthonormal group} $O\left(m_1,\dots,m_\kappa,A\right)$ is the subgroup of $\gl{n}{A}$ formed by all $T\in A^{n\times n}$ satisfying the following conditions:
\begin{enumerate}
\item $\ T_{ii}\in\gl{m_i}{A},\ 1\le i\le \kappa$,
\item $\ T_{ij}\in \m^{m_i\times m_j}$ for all $i>j$.
\end{enumerate}
\end{definition}

\begin{theorem}\label{transition}
Let $0=\eps_1< \eps_2<\cdots<\eps_\kappa<1$ be the elements in the underlying set of $W_{L/K}$, ordered by increasing size.

For $1\le i\le \kappa$ denote $m_i=f_{\eps_i}$, so that $W_{L/K}=\{\eps_i^{m_i}\mid 1\le i\le \kappa\}$.

Let $\mathbb{B}\in B^n$ be an orthonormal basis of $L/K$. Then, $\mathbb{B}'\in L^n$ is an orthonormal basis if and only if  
the transition matrix from $\mathbb{B}$ to $\mathbb{B'}$ belongs to the orthonormal group
 $O\left(m_1,\dots,m_\kappa,A\right)$.
\end{theorem}

\begin{proof}
Write $\mathbb{B}=(\alpha_1,\dots\alpha_n)$ and denote
$$n_0=0;\quad n_i=m_1+\cdots+m_i,\quad 1\le i\le \kappa.$$ 

For a given $T\in O\left(m_1,\dots,m_\kappa,A\right)$ take
\begin{equation}\label{T}
\begin{pmatrix}\alpha'_1\\\vdots\\\alpha'_n \end{pmatrix}=
T\begin{pmatrix}\alpha_1\\\vdots\\\alpha_n \end{pmatrix}.
\end{equation}

For a given index $1\le d\le n$, let $1\le i\le \kappa$ be determined by $n_{i-1}< d \le n_i$. From (\ref{T}) we deduce that $\alpha'_d=\beta_-+\beta+\beta_+$, where 
\begin{equation}\label{tridec}
\as{1.2}
\begin{array}{ll}
\beta_-=\sum_{j=1}^{n_{i-1}}a_j\alpha_j &\mbox{ with  }a_j\in \m,\\
\beta=\sum_{j=n_{i-1}+1}^{n_i}a_j\alpha_j&\mbox{ with $a_j\in A$, not all of them in } \m,\\
\beta_+=\sum_{j=n_i+1}^na_j\alpha_j&\mbox{ with  }a_j\in A,
\end{array}
\end{equation}
being $(a_1\cdots a_n)$ the $d$-th row of $T$.
Since the family $\mathbb{B}$ is reduced, we deduce $$w(\beta_-)>1,\quad w(\beta)=\eps_i,\quad w(\beta_+)>\eps_i.$$ Hence, $w(\alpha'_d)=\eps_i=w(\alpha_d)$, and 
$$
\rd_{\eps_i}(\alpha'_d)=\rd_{\eps_i}(\beta)=\sum_{j=n_{i-1}+1}^{n_i}\overline{a}_j\rd_{\eps_i}(\alpha_j).
$$ 
Thus, $T$ preserves the sequence of $w$-values and moreover: 
\begin{equation}\label{redeps}
\begin{pmatrix}\rd_{\eps_i}(\alpha'_{n_{i-1}+1})\\\vdots\\\rd_{\eps_i}(\alpha'_{n_i}) \end{pmatrix}=
\overline{T}_{ii}\begin{pmatrix}\rd_{\eps_i}(\alpha_{n_{i-1}+1})\\\vdots\\\rd_{\eps_i}(\alpha_{n_i}) \end{pmatrix},\quad 1\le i\le\kappa.
\end{equation}
By Theorem \ref{criterion}, the family $\mathbb{B}'=(\alpha'_1,\dots\alpha'_n)\in B^n$ is reduced too, and by Lemma \ref{reducedbasis} it is an orthonormal basis.

Conversely, suppose that  $\mathbb{B}'=(\alpha'_1,\dots\alpha'_n)\in B^n$ is an orthonormal basis of $L/K$, and let $T\in\gl{n}{A}$ be the transition matrix from $\mathbb{B}$ to $\mathbb{B}'$, determined by (\ref{T}). From $w(\alpha'_d)=w(\alpha_d)$ we deduce (\ref{tridec}) and (\ref{redeps}). This proves that $T$ belongs to the orthogonal group.
\end{proof}

\section{Triangular reduced integral bases}\label{secTriang}
We are interested in the computation of \emph{triangular} reduced integral bases, because they are useful in many practical applications. For instance, they facilitate the computation of global integral bases by patching local ones with the aid of the chinese remainder theorem.

\begin{definition}\label{triang}
We say that $(\alpha_0,\dots,\alpha_{n-1})\in L^n$ is a \emph{triangular family} if $\alpha_j=g_j(\t)\pi^{r_j}$ for a certain monic polynomial $g_j\in A[x]$ of degree $j$, and an integer $r_j$, for each $0\le j<n$.
\end{definition}

In other words, the transition matrix $T\in\gl{n}{K}$ determined by  
$$
\begin{pmatrix}\alpha_{n-1}\\\vdots\\\alpha_0 \end{pmatrix}=T\left(\begin{array}{l} \t^{n-1}\!\!\!\\\vdots\\1\!\!\!\end{array}\right).
$$is upper triangular with entries $\pi^{r_{n-1}},\dots,\pi^{r_0}$ at the diagonal.

By Theorem \ref{reducedtriang} below,  the computation of triangular reduced integral bases amounts to compute, for each $0\le j<n$, a monic polynomial $g_j\in A[x]$ of degree $j$ such that $g_j(\t)$ attains the maximal $w$-value among all monic polynomials in $A[x]$ of degree $j$.

\begin{definition}\label{optimal}
For $0\le j<n$, consider:
$$
\delta_j=\mx\{w(g(\t))\mid g\in A[x] \mbox{ monic of degree }j\}.
$$
Since the valuations $w_1,\dots,w_t$ are discrete, this maximal value is attained by some monic polynomial $g\in A[x]$. In other words, $\delta_j\in w(L^*)$ for all $j$. 

We denote by $m(\delta_j)$ the multiplicity of $\delta_j$ in the family $\delta_0,\dots, \delta_{n-1}$. 
\end{definition}

Clearly, $\delta_0\le\cdots\le \delta_{n-1}$. In fact, if $0<j<n$ and $g\in [x]$ is a monic polynomial of degree $j-1$ with $w(g(\t))=\delta_{j-1}$, we have
$$
\delta_j\ge w(\t g(\t))\ge w(\t)+w(g(\t))\ge \delta_{j-1}.
$$

The following result proves the existence of triangular reduced integral bases, and offers an interesting point of view to distinguish triangular reduced integral bases among triangular integral bases. 

\begin{theorem}\cite[Thm. 1.4]{hayden}\label{reducedtriang}
Let $g_0,\dots,g_{n-1}\in A[x]$ be monic polynomials of degree $0,\dots,n-1$, respectively.
Let $\nu_j=w(g_j(\t))$ for $0\le j<n$, and consider the set
$\bb=\left\{g_0(\t)\pi^{-\nu_0},\dots, g_{n-1}(\t)\pi^{-\nu_{n-1}}\right\}$.
Then, 
\begin{enumerate}
\item $\bb$ is an integral basis if and only if $\lfloor \nu_j\rfloor=\lfloor\delta_j\rfloor$ for $0\le j<n$. 
\item $\bb$ is a reduced integral basis if and only if $ \nu_j=\delta_j$ for $0\le j<n$.
\end{enumerate}
\end{theorem}

By Theorems \ref{distribution} and \ref{reducedtriang}, the multiset $\{\delta_j+\Z\mid 0\le j<n\}$ is an intrinsic invariant of the extension $L/K$. More precisely, $$W_{L/K}=\{\delta_0-[\delta_0],\dots,\delta_{n-1}-[\delta_{n-1}]\}.$$
However, the multiset $\Delta=\{\delta_0,\dots,\delta_{n-1}\}$ of all maximal $w$-values depends on the choice of the  polynomial $f$ defining the extension $L/K$ (but not on the choice of the root $\t$ of $f$).

Consider Gauss' extension of the valuation $v$ to the polynomial ring $A[x]$:
$$
v\left(\sum\nolimits_{0\le d}a_dx^d\right)=\mn\left\{v(a_d)\mid 0\le d\right\}.
$$

\begin{lemma}\label{wpoly}
For a given $g=\sum\nolimits_{d=0}^{n-1}a_dx^d\in A[x]$, let $d_0$ be maximal with the property
$v(g)=v(a_{d_0})$. Then, $w(g(\t))\le v(g)+\delta_{d_0}$. 
\end{lemma}

\begin{proof}
Let $g_0,\dots,g_{n-1}\in A[x]$ be monic polynomials of degree $0,\dots,n-1$ attaining the maximal $w$-values $\delta_0,\dots,\delta_{n-1}$. Obviously, we can write $g$ in a unique way as $g=\sum_{d=0}^{n-1}b_dg_d$ with $b_0,\dots,b_{n-1}\in A$. By hypothesis, 
$$
v(a_{n-1}),\dots, v(a_{d_0+1})>v(g),\quad  v(a_{d_0})=v(g).
$$
Clearly, this forces the coefficients $b_d$ to satisfy the same conditions: 
$$
v(b_{n-1}),\dots, v(b_{d_0+1})>v(g),\quad  v(b_{d_0})=v(g).
$$
By Theorem \ref{reducedtriang} and Lemma \ref{scale}, the family $g_0(\t),\dots,g_{n-1}(\t)$ is reduced, so that $w(g(\t))=\mn\{v(b_d)+\delta_d\mid 0\le d<n\}\le v(b_{d_0})+\delta_{d_0}= v(g)+\delta_{d_0}$.   
\end{proof}

\subsection{Triangulation of reduced integral bases}
In this section, we discuss a triangulation procedure which may be applied to any reduced integral basis $\bb=\{\beta_1,\dots,\beta_n\}$ of the form
$$
\beta_j=g_j(\t)\pi^{-\lfloor \nu_j\rfloor},\quad \nu_j=w(q_j(\t)),\quad 1\le j\le n,
$$
where $q_1\dots,q_n$ are polynomials in $A[x]$ whose $w$-values $\nu_1,\dots,\nu_n$ are known. 

Such a basis is provided, for instance, by the \emph{method of the quotients} \cite{bases}, or the \emph{multipliers method} \cite{bauch1}, both based on the Montes algorithm \cite{GMN,algo}. 

The standard triangulation procedures, like the Hermite Normal Form (HNF) routine, destroy reduceness. Our aim is to use these standard techniques but in a controlled way which preserves reduceness. 

\begin{definition}\label{qfamily}
For an integer $0<d\le n$, we say that $q_1,\dots,q_d\in A[x]$ is a \emph{$d$-reduced polynomial family} if the following conditions are satisfied:
\begin{enumerate}
\item $\deg(q_j)<d$ and $v(q_j)=0$, for all $1\le j\le d$.
\item $q_1(\t),\dots,q_d(\t)$ is a reduced family.
\end{enumerate}
\end{definition}

\begin{lemma}\label{flag}
Consider the flag of $K$-subspaces of $L$: 
$$
L=L_n\supsetneq L_{n-1}\supsetneq \cdots\supsetneq L_1=K\supsetneq L_0=\{0\},
$$
where $L_i=\langle 1,\t,\dots,\t^{i-1}\rangle_K$ for $1\le i\le n$.

For $0<d\le n$, let $q_1,\dots,q_d$ be a $d$-reduced polynomial family, with $w$-values $\nu_1,\dots,\nu_d$. Then, 
\begin{enumerate}
\item $\nu_j\le \delta_{d-1}$ for all $1\le j\le d$.
\item $q_1(\t)/\pi^{\lfloor\nu_1\rfloor},\dots,q_d(\t)/\pi^{\lfloor\nu_d\rfloor}$
is an $A$-basis of the $A$-module $B\cap L_d$.
\end{enumerate}
\end{lemma}

\begin{proof}
The first item follows immediately from Lemma \ref{wpoly}. The second follows from the same arguments of the proof of Lemma \ref{reducedbasis}.
\end{proof}

The triangulation procedure iterates certain \emph{triangulation steps} with the following aim:\medskip

\noindent{\bf Triangulation step}\medskip

\noindent{\bf Input.} A $d$-reduced polynomial family $q_1,\dots,q_d$, whose sequence of $w$-values 
$\nu_1,\dots,\nu_d$ is known.\medskip

\noindent{\bf Output.} 
\begin{itemize}
\item $\delta_{d-1}$ and $m:=m(\delta_{d-1})$.
\item Monic polynomials $g_{d-1},\dots,g_{d-m}\in A[x]$ such that
$$\deg (g_{d-j})=d-j,\quad w(g_{d-j}(\t))=\delta_{d-1},\quad 1\le j\le m.$$
\item A $(d-m)$-reduced polynomial family  $q'_1,\dots,q'_{d-m}$, whose sequence of $w$-values 
$\nu'_1,\dots,\nu'_{d-m}$ is known.
\end{itemize}\medskip

We start with the $n$-reduced poynomial family provided by either method, \emph{quotients} \cite{bases} or \emph{multipliers} \cite{bauch1}. Let $r$ be the number of pairwise different maximal $w$-values.
After $r$ triangulation steps, we end with a family of monic polynomials $g_{n-1},\dots,g_0\in A[x]$ of degree $n-1,\dots,0$, attaining the maximal $w$-values $\delta_{n-1},\dots,\delta_0$. By Theorem \ref{reducedtriang}, this yields a triangular reduced integral basis of $L/K$. 

From now on, we fix a $d$-reduced polynomial family $q_1,\dots,q_d$ with sequence of $w$-values $\nu_1,\dots,\nu_d$. Let $\nu=\mx\{\nu_j\mid 1\le j\le d\}$ and let $\ell$ be the multiplicity of $\nu$ in the sequence $\nu_1,\dots,\nu_d$. 

We suppose moreover that the polynomials are ordered so that: 
$$
\nu_1=\cdots=\nu_\ell=\nu,\quad \nu_{\ell+1},\dots,\nu_d<\nu.
$$

The concrete description of the triangulation step, given in Proposition \ref{step}, requires an auxiliary result.

\begin{lemma}\label{P}
Let $\Gamma_{\ell,d}(A)$ be the subgroup of $\gl{d}{A}$ of all matrices $U$ of the form:
\begin{equation}\label{TPQ}
\as{1.2}U=\left(\begin{array}{c|c}
P&0\\\hline Q&I_{d-\ell}
\end{array}\right),\quad P\in\gl{\ell}{A}, \ Q\in A^{(d-\ell)\times \ell}.
\end{equation}
For any $U\in\Gamma_{\ell,d}(A)$ the polynomials $q'_1,\dots,q'_d\in A[x]$ obtained as
$$
\begin{pmatrix}q'_1\\\vdots\\q'_d \end{pmatrix}=U
\begin{pmatrix}q_1\\\vdots\\q_d \end{pmatrix}
$$
yield a reduced family $q'_1(\t),\dots,q'_d(\t)\in L$ with the same sequence of $w$-values 
$\nu_1,\dots,\nu_d$.
\end{lemma}

\begin{proof}


The sequence of $w$-values is preserved by an argument completely analogous to that used in the proof of Theorem \ref{transition}.

For all $1\le j\le d$, let us denote
$$\beta_j=q_j(\t)/\pi^{\lfloor\nu_j\rfloor},\quad \beta'_j=q'_j(\t)/\pi^{\lfloor\nu_j\rfloor},\quad \eps_j=\nu_j-\lfloor \nu_j\rfloor=w(\beta_j)=w(\beta'_j).$$

Consider the sets $\bb=\{\beta_j\mid 1\le j\le d\}$, $\bb'=\{\beta'_j\mid 1\le j\le d\}$.
By Lemma \ref{scale}, $\bb$ is reduced and we need only to check that $\bb'$ is reduced. We shall prove this by applying the reduceness criterion of Theorem \ref{criterion}. 


Let $U$ be given by the matrices $P$, $Q$ as in (\ref{TPQ}). For $j>\ell$, we have
$$
q'_j=q_j+a_1q_1+\cdots+ a_\ell q_\ell,
$$
with $a_1,\dots,a_\ell\in A$ the entries in the $(j-\ell)$-th row of $Q$. Since $$w(a_1q_1+\cdots+ a_\ell q_\ell)\ge\nu>\nu_j,$$ we deduce that $w(\beta'_j-\beta_j)>\eps_j$. This implies $\rd_{\eps_j}(\beta'_j)=\rd_{\eps_j}(\beta_j)$.

On the other hand, let $\eps=\nu-\lfloor \nu\rfloor=w(\beta_j)=w(\beta'_j)$ for all $1\le j\le \ell$. The mapping $\rd_\eps$  is linear in the following sense:
$$
\begin{pmatrix}\beta'_1\\\vdots\\\beta'_{\ell} \end{pmatrix}=
P\begin{pmatrix}\beta_1\\\vdots\\\beta_{\ell} \end{pmatrix}
\quad\Longrightarrow\quad
\begin{pmatrix}\rd_\eps(\beta'_1)\\\vdots\\\rd_\eps(\beta'_{\ell}) \end{pmatrix}=
\overline{P}\begin{pmatrix}\rd_\eps(\beta_1)\\\vdots\\\rd_\eps(\beta_\ell) \end{pmatrix}.
$$

Consider the set $I=\{1\le j\le n\mid\eps_j=\eps\}$, which contains $1,\dots,\ell$ and some more indices. By definition,
$$\bb_\epsilon=\{\beta_j\mid j\in I\}, \quad \bb'_\epsilon=\{\beta'_j\mid j\in I\}.$$ 
By Theorem \ref{criterion}, $\rd_\eps(\bb_\eps)$ is a $k$-linearly independent subset of $V$. Since $\overline{P}\in \gl{\ell}{k}$, the family $\rd_\eps(\beta'_1),\dots,\rd_\eps(\beta'_\ell)$ is $k$-linearly independent and generates the same subspace than $\rd_\eps(\beta_1),\dots,\rd_\eps(\beta_\ell)$. Since  $\rd_{\eps}(\beta'_j)=\rd_{\eps}(\beta_j)$ for all $j\in I$, $j>\ell$, the set $\rd_\eps(\bb'_\eps)$ is $k$-linearly independent too. 

Also, for all $\delta\in w(L^*)\cap[0,1)$, $\delta\ne\eps$, the set
 $\rd_\delta(\bb'_\delta)=\rd_\delta(\bb_\delta)$ is $k$-linearly independent. By Theorem \ref{criterion}, $\bb'$ is reduced.
\end{proof}

Let $T=(t_{i,j})\in A^{d\times d}$ be the matrix whose $i$-th row captures the coefficients of the polynomial $q_i$ in decreasing degree. Thus,
\begin{equation}\label{encoding}
T\left(\begin{array}{l} x^{d-1}\!\!\!\\\vdots\\1\!\!\end{array}\right)=\begin{pmatrix}q_1\\\vdots\\q_d \end{pmatrix}.
\end{equation}

We say that the rows of $T$ \emph{encode} the polynomials $q_1,\dots,q_d$.

The triangulation step replaces the matrix $T$ with $UT$ for an adequate $U$ in the group $\Gamma_{\ell,d}(A)$ introduced in Lemma \ref{P}, and then divides out the rows of $UT$ by adequate powers of $\pi$. 

\begin{proposition}\label{step}
Let $T_{\op{up}}$ and  $T_{\op{down}}$ be the matrices formed by the first $\ell$ rows of $T$ and the last $d-\ell$ rows of $T$, respectively. Express the Hermite normal form of $T_{\op{up}}$ as:
$$\as{1.2}
\left(\begin{array}{c|c}
I_m&\qquad C\qquad \\\hline 0&\qquad D\qquad 
\end{array}\right),\quad C\in A^{m\times (d-m)},\ D\in A^{(\ell-m)\times (d-m)},
$$
with all entries in the first column of $D$ belonging to $\m$. Write
$$
T_{\op{down}}=(E\mid\quad F\quad),\quad E\in A^{(d-\ell)\times m},\ F\in A^{(d-\ell)\times (d-m)},
$$
$$
\as{1.2}
T'=\left(\begin{array}{c}
D\\\hline \ F-EC 
\end{array}\right)\in A^{(d-m)\times (d-m)}.
$$
Let $h_1,\dots,h_{d-m}$ be the polynomials encoded by the rows of $T'$.  Then, 
\begin{enumerate}
\item $\delta_{d-1}=\nu:=\mx\{\nu_j\mid 1\le j\le d\}$ and $m(\delta_{d-1})=m$.
\item The monic polynomials $g_{d-1},\dots,g_{d-m}$ encoded by the rows of the matrix $(I_m\mid \ C\ )$ have degree $d-1,\dots,d-m$ and satisfy$$w(g_{d-1}(\t))=\cdots=w(g_{d-m}(\t))=\delta_{d-1}.$$
\item All entries in the matrix $D$ belong to $\m$.
\item The polynomials $q'_j=h_j\,\pi^{-v\left(h_j\right)}$, for $1\le j\le d-m$,  form a $(d-m)$-reduced polynomial family whose sequence of $w$-values is: $$\nu'_j=\nu_{m+j}-v\left(h_j\right),\quad 1\le j\le d-m.$$
\end{enumerate}
\end{proposition}

\begin{proof}
Let $g\in A[x]$ be a monic polynomial of degree $d-1$ such that $w(g(\t))=\delta_{d-1}$. By item (2) of Lemma \ref{flag}, we can write $g=\sum_{j=1}^da_jq_j$ with $a_j\in K$. By reduceness,
$$
\delta_{d-1}=w(g(\t))=\mn\{v(a_j)+\nu_j\mid 1\le j\le d\}.
$$
By Lemma \ref{flag}, $\nu_j\le \delta_{d-1}$ for all $j$, so that $\nu\le \delta_{d-1}$. We have,
\begin{equation}\label{chain}
\nu_j\le\delta_{d-1}\le v(a_j)+\nu_j,\quad 1\le j\le d.
\end{equation}
This implies $v(a_j)\ge0$ for all $j$.
Since $g$ is monic, we have necessarily $v(a_{j_0})=0$ for some index $j_0$. From (\ref{chain}) we deduce $\delta_{d-1}= \nu_{j_0}$, and this implies that $\delta_{d-1}=\nu$.

Consider the following transformation of the matrix $T$ by multiplication on the left by a matrix in the group $\Gamma_{\ell,d}(A)$:
$$
\as{1.2}\left(\begin{array}{c|c}
P&0\\\hline 0&I_{d-\ell}\end{array}\right)\left(\begin{array}{c}
T_{\op{up}}\\\hline T_{\op{down}}\end{array}\right)=\left(\begin{array}{c}
PT_{\op{up}}\\\hline T_{\op{down}}\end{array}\right)=
\left(\begin{array}{c|c}
I_m&\qquad C\qquad \\\hline 0&\qquad D\qquad \\\hline E&\qquad F\qquad \end{array}\right),
$$
where $P\in\gl{\ell}{A}$ reflects the Gaussian elimination transformations that were applied to compute the Hermite normal form of $T_{\op{up}}$.

Now, let $E'=(E\mid 0)$, where $0$ indicates the null matrix in $A^{(d-\ell)\times (\ell-m)}$. We apply a further transformation by an element in $\Gamma_{\ell,d}(A)$:
$$
\as{1.2}\left(\begin{array}{c|c}
I_\ell&0\\\hline -E'&I_{d-\ell}\end{array}\right)
\left(\begin{array}{c|c}
I_m&\qquad C\qquad \\\hline 0&\qquad D\qquad \\\hline E&\qquad F\qquad \end{array}\right)=
\left(\begin{array}{c|c}
I_m&\quad\ C\quad\ \\\hline 0&\quad D\quad \\\hline 0& F-EC \end{array}\right)=:M.
$$

Let $g_{d-1},\dots,g_{d-m},h_1,\dots,h_{d-m}$ be the polynomials encoded by the rows of $M$. By Lemma \ref{P}, $g_{d-1}(\t),\dots,g_{d-m}(\t),h_1(\t),\dots,h_{d-m}(\t)$ is a reduced family with sequence of $w$-values $\nu_1,\dots,\nu_d$. This proves (2) and (4). 

In order to prove that $m=m(\delta_{d-1})$, it suffices to show that a monic polynomial $g\in A[x]$ of degree $d-m-1$ has necessarily $w(g(\t))<\delta_{d-1}$. 
Suppose $w(g(\t))\ge \delta_{d-1}$ for such a polynomial $g$, and let us show that this leads to a contradiction.

Since 
 $h_1(\t),\dots,h_{d-m}(\t)\in L_{d-m}$ are $K$-linearly independent elements (by reduceness), they form a $K$-basis of $L_{d-m}$. Hence, we may write 
 $$
 g(\t)=a_1h_1(\t)+\cdots +a_{d-m}h_{d-m}(\t), \quad a_1,\dots,a_{d-m}\in K.
 $$
 By reduceness, 
 $$
\delta_{d-1}\le w(g(\t))\le v(a_j)+\nu_{m+j},\quad 1\le j\le d-m.
$$
This implies $w(a_j)\ge0$ for all $1\le j\le \ell-m$ (because $\nu_{m+j}=\nu=\delta_{d-1}$), and 
 $w(a_j)>0$ for all $\ell-m< j\le d-m$ (because $\nu_{m+j}<\nu=\delta_{d-1}$). On the other hand, the coefficients of degree $d-m-1$ of $h_1,\dots,h_{\ell-m}$ belong to $\m$, because they form the first column of $D$. Hence, the leading coefficient of $g$ belongs to $\m$, and this contradicts the fact that $g$ is monic. This ends the proof of item (1).
 
 Item (3) is a consequence of Lemma \ref{wpoly}. If $q\in A[x]$ is the polynomial encoded by any row of $D$, we have $\delta_{d-1}=w(q(\t))\le v(q)+\delta_{d_0}$ for some $d_0\le \deg(q)<d-m$. We deduce that $v(q)\ge \delta_{d-1}-\delta_{d_0}>0$. Hence, all coefficients of $q$ belong to $\m$.  
\end{proof}

\section{Reduced normal form}\label{secRNF}
Let $\ut$ be the unitriangular group; that is, the subgroup of $\gl{n}{A}$ of all upper triangular matrices with $1$'s at the diagonal.

The triangulation procedure of section \ref{secTriang} computes a matrix $T=(t_{ij})\in \ut$ whose rows encode a family of monic polynomials $g_{n-1},\dots,g_0$ such that  $g_{n-1}(\t),\dots,g_0(\t)$ attain the maximal $w$-values $\delta_{n-1},\dots,\delta_0$.

The aim of this section is to apply further simplifications to the entries above the main diagonal of $T$ to obtain a canonical form, still encoding a family of polynomials attaining the maximal $w$-values. By Theorem \ref{reducedtriang}, this is the only condition we need to ensure that 
\begin{equation}\label{btr}
\bb=\{1,\,g_1(\t)/\pi^{\lfloor\delta_1\rfloor},\dots,\,g_{n-1}(\t)/\pi^{\lfloor\delta_{n-1}\rfloor}\}
\end{equation}
is still a triangular reduced integral basis. 

For each positive integer $d$ choose $\rr_d\subset A$ a set of representatives of the classes modulo $\m^d$. 

\begin{definition}
A triangular reduced integral basis $\bb$ as in (\ref{btr}) is said to be in \emph{reduced normal form} (RNF) if the matrix $T=(t_{ij})\in \ut$ whose rows encode the family $g_{n-1},\dots,g_0$ satisfies $t_{ij}\in\rr_{\lceil \delta_{n-i}-\delta_{n-j}\rceil}$ for all $i<j$. In this case, we also say that the matrix $T$ is in RNF.
\end{definition}

The condition for $\bb$ to be in HNF is $t_{ij}\in\rr_{\lfloor \delta_{n-i}\rfloor-\lfloor\delta_{n-j}\rfloor}$ for all $i<j$. For each pair of indices $1\le i<j\le n$, we have $\lfloor \delta_{n-i}\rfloor-\lfloor\delta_{n-j}\rfloor\le \lceil \delta_{n-i}-\delta_{n-j}\rceil$, and equality holds if and only if $\delta_{n-i}-\delta_{n-j}\in\Z$.  

Therefore, for the pairs of indices $i<j$ such that $\delta_{n-i}-\delta_{n-j}\not\in\Z$ the condition on $t_{ij}$ for $T$ to be in RNF is weaker than the condition for $T$ to be in HNF.

\begin{lemma}\label{preserve}
For $i<j$ and $a\in A$, consider the monic polynomial $g=g_j-ag_i\in A[x]$ of degree $j$. Then, $g$ keeps the maximal $w$-value $w(g(\t))=\delta_j$ if and only if 
$v(a)\ge\delta_j-\delta_i$.
\end{lemma}

\begin{proof}
By reduceness, $w(g(\t))=\mn\{\delta_j,v(a)+\delta_i\}$. 
\end{proof}

Lemma \ref{preserve} shows that there is a unique triangular reduced integral basis of $L/K$ in RNF, for a given defining polynomial $f$ of the extension $L/K$. Also, this lemma provides a concrete procedure to compute the RNF once a triangular reduced integral basis is given. 

Actually, for a triangular reduced basis obtained by the triangulation routine of section \ref{secTriang}, we may use a blockwise procedure to obtain the RNF.

Let $\rho_1>\cdots>\rho_r$ be the ordered sequence of pairwise different elements in the multiset $\Delta=\{\delta_0,\dots,\delta_{n-1}\}$. Let $m_1,\dots,m_r$ be the corresponding multiplicities, so that $\Delta=\{\rho_1^{m_1},\dots,\rho_r^{m_r}\}$. 

Suppose that $T\in\ut$ encodes the numerators of a triangular reduced integral basis obtained by the triangulation procedure of section \ref{secTriang}. Let  
$T=(T_{ij})_{1\le i,j\le r}$ be the block decomposition of $T$ induced by the partition $n=m_1+\cdots+m_r$. Note that $T_{ii}=I_{m_i}$ for all $i$, and $T_{ij}=0$ for all $i>j$.\medskip

\noindent{\bf RNF routine}\vskip 0.1cm

\noindent{\bf Input.} $T=(T_{ij})_{1\le i,j\le r}\in\ut$ and the list $\rho_1>\dots>\rho_r$ of maximal $w$-values.\medskip

1. for $i=1$ to $r-1$ do

2. \qquad for $j=i+1$ to $r$ do

3. \qquad\qquad express $T_{ij}=C+\pi^{\lceil \rho_i-\rho_j\rceil}D$, with $C\in \rr_{\lceil \rho_i-\rho_j\rceil}^{m_i\times m_j}$

4. \qquad\qquad for $k=j$ to $r$ do

5. \qquad\qquad\qquad $T_{ik}\leftarrow T_{ik}-\pi^{\lceil \rho_i-\rho_j\rceil}DT_{jk}$ \medskip
 
\noindent{\bf Output.} A matrix $T\in\ut$ in RNF.

\section{Computational implications. An example}\label{secComp}
In this section, we discuss the practical computation of integral bases in reduced normal form, and we exhibit an example.

Let $A_v\subset K_v$ be the completion of $A\subset K$ with respect to the $v$-adic topology. Let us still denote by $v\colon \overline{K}_v \to \Q\cup\{\infty\}$ the canonical (non-discrete) extension of the valuation $v$ to a fixed algebraic closure of $K_v$.

Let $f=F_1\cdots F_t$ be the factorization of $f$ into a product of irreducible factors in $A_v[x]$. These factors $F_1,\dots, F_t$ are in 1-1 correspondence with the extensions $w_1,\dots,w_t$ of $v$ to $L$. In fact, each $F_i$ determines a finite field extension $L_i/K_v$, and the field $L$ may be embedded into $L_i$ by sending $\t$ to a root of $F_i$ in $ \overline{K}_v$. The valuation $w_i$ is obtained as the composition 
$$w_i\colon L\hookrightarrow L_i\hookrightarrow \overline{K}_v\stackrel{v}\lra \Q\cup\{\infty\}.$$

The \emph{method of the quotients} introduced in \cite{bases} computes a reduced integral basis as a by-product of the Montes algorithm \cite{algo,GMN}, which is a kind of polynomial factorization routine over $A_v[x]$. 

Bauch \cite{bauch1} and Stainsby \cite{hayden}, found independent algorithms, called \emph{multipliers} and \emph{MaxMin} respectively, which compute reduced integral bases as an application of the Montes algorithm in combination with the Single Factor Lifting algorithm (SFL) \cite{SFL}. The MaxMin algorithm has the advantage of computing directly triangular reduced integral bases. 

For each irreducible factor $F_i$, the Montes algorithm computes 
a family of monic polynomials $\phi_1,\dots,\phi_r,\phi_{r+1}\in A[x]$, where $r$ is the \emph{Okutsu depth} of $F_i$, such that the list $[\phi_1,\dots,\phi_r]$ is an \emph{Okutsu frame} of $F_i$, and $\phi_{r+1}$ is an \emph{Okutsu approximation} to $F_i$ (cf. \cite{Ok,okutsu}). This means that $\phi_{r+1}$ is ``sufficiently close" to $F_i$ in the $v$-adic topology.  

If $m_\ell=\deg\phi_\ell$ for $1\le \ell\le r+1$, then
$$
0<m_1<\cdots<m_r<m_{r+1}=\deg (F_i),\quad m_1\mid\cdots \mid m_{r+1}.
$$

The Okutsu frame $[\phi_1,\dots,\phi_r]$ is a family of polynomials with maximal $w_i$-values according to their degree. More precisely, for any monic polynomial $g\in A[x]$ we have
$$
\deg(g)<m_{\ell+1}\mbox{ for } 0\le\ell\le r\ \Longrightarrow\ \dfrac{w_i(g(\t))}{\deg(g)}\le \dfrac{w_i(\phi_\ell(\t))}{m_\ell}, 
$$
where we agree that $\phi_0=1$. The polynomials in the Okutsu frames of all factors $F_1\dots,F_t$ will be simply called \emph{$\phi$-polynomials}.


We may summarize two methods for the computation of integral bases in reduced normal form as follows.\medskip

\noindent{\bf Quotients}\vskip 0.1cm\it

(Q1) Apply the Montes algorithm to compute an Okutsu frame of each $F_i$, but skip the computation of the Okutsu approximations.\vskip 0.1cm

(Q2) Compute an $n$-reduced polynomial family $q_1,\dots,q_n$ and the corresponding sequence of $w$-values $\nu_1,\dots,\nu_n$. \vskip 0.1cm

(Q3) Apply the triangulation routine of section \ref{secTriang}.\vskip 0.1cm

(Q4) Apply the RNF routine of section \ref{secRNF}.\rm\medskip

Each polynomial $q_i$ is an adequate product of the quotients of certain divisions with remainder of $f$ by powers of $\phi$-polynomials, performed (and stored) along the execution of the Montes algorithm (cf. \cite{bases}).\medskip

\noindent{\bf MaxMin}\vskip 0.1cm\it

(MM1) Apply the full Montes algorithm to compute Okutsu frames and Okutsu approximations of all $F_i$. \vskip 0.1cm

(MM2) Compute the maximal $w$-values $\delta_0,\dots,\delta_{n-1}$ and formal expressions of monic polynomials $g_0,\dots,g_{n-1}$ attaining these values, as products of $\phi$-polynomials and Okutsu approximations. \vskip 0.1cm

(MM3) Apply the SFL routine to improve the Okutsu approximations to a precision determined by the computations of (MM2).\vskip 0.1cm

(MM4) Execute the products indicated formally in (MM2) and compute the polynomials $g_0,\dots,g_{n-1}$, yielding a triangular reduced integral basis.\vskip 0.1cm

(MM5) Apply the RNF routine of section \ref{secRNF}.\rm\medskip

Steps (Q2) and (MM4) have the same cost: $O(n)$ multiplications in $A[\t]$.
Steps (Q4) and (MM5) have the same cost too. On the other hand, step (MM2) has a negligible cost.  

Therefore, in  order to compare the computational performance of the two methods, we must compare the cost of the triangulation routine (Q3) with the extra tasks of MaxMin: computation of the Okutsu approximations (part of (MM1)) and their improvements (MM3). 

Now, the complexity of the steps (MM1)-(MM4) \cite[Thm. 3.5]{hayden} is lower than the complexity of Gaussian elimination, which requires $O(n^3)$ multiplications in $A$. 

Thus, for $n$ large, MaxMin is much faster than Quotients, or the similar algorithm resulting from the use of the multipliers method. For $n$ of a moderate size, say $n<100$, the two methods have a similar performance for randomly chosen inputs. Therefore, we may conclude that MaxMin is the reasonable choice as a prototype algorithm for the computation of integral bases in RNF.

Also, MaxMin and Multipliers have the advantage of being able to compute reduced bases of fractional ideals of $B$, while Quotients is only able to compute the maximal order $B$ itself.

\subsection{An example}
We end this section with a concrete example.

Let $A=\Z_{(2)}$ be the localization of $\Z$ at the prime ideal $2\Z$. Thus, $K=\Q$ and the valuation $v$ of $A$ is the ordinary $2$-adic valuation. For each positive integer $d$, let us take $\rr_d=(-2^{d-1},2^{d-1}]\cap \Z$.

Consider the number field $L=\Q(\t)$, where $\t$ is a root of the monic irreducible polynomial:
$$
f=x^8-x^7+21x^6-20x^5-368x^4+388x^3-516x^2+128x+128\in A[x].
$$

The Montes algorithm determines the prime ideal decomposition: 
$$
2B=\p_1^4\,\p_2^2\,\p_3,\quad f_1=f(\p_1/2)=f_2=f(\p_2/2)=1,\ f_3=f(\p_3/2)=2,
$$
so that  $f=F_1F_2F_3$ has three irreducible factors in $\Z_2[x]$. The algorithm finds the following Okutsu frames and Okutsu approximations, too:
$$
\begin{array}{cl}
\mbox{$[x,\,x^2+2x+2]$},&\phi_{\p_1}=x^4+4x^3+8x^2+16x+4\approx F_1,\\
\mbox{$[ x ]$}, & \phi_{\p_2}=x^2+32\approx F_2, \\ 
\mbox{[\ ]}, & \phi_{\p_3}=x^2+x+1\approx F_3.
\end{array}
$$
The irreducible factor $F_3$ is irreducible modulo $2$. Hence, it has Okutsu depth zero and its Okutsu frame is an empty list.

The reduced integral basis computed by the method of the quotients is:
$$
T_{\op{Quotients}}=
\begin{pmatrix}
1&31&21&12&16&4&28&0\\
9&7&11&2&14&4&0&0\\
0&1&5&1&0&6&0&0\\
0&1&7&5&4&0&4&4\\
1&2&3&2&1&0&0&0\\
0&0&0&1&3&1&0&0\\
0&1&0&0&0&0&0&0\\
1&0&0&0&0&0&0&0
\end{pmatrix},
\qquad
\vec{\nu}=\begin{pmatrix}
9/2\\
13/4\\
11/4\\
2\\
3/2\\
1\\
0\\
0
\end{pmatrix}.
$$
This matrix $T_{\op{Quotients}}$ encodes a family of polynomials $q_1,\dots,q_8\in A[x]$ as indicated in (\ref{encoding}). The column $\vec{\nu}$ contains the corresponding sequence of $w$-values:  $\nu_1=w(q_1(\t)),\dots,\nu_8=w(q_8(\t))$. Recall that the corresponding reduced integral basis is 
$$
\bb=\left\{q_1(\t)/2^{\lfloor\nu_1\rfloor},\dots,q_8(\t)/2^{\lfloor\nu_8\rfloor}\right\}.
$$
In agreement with Theorem \ref{distribution}, $w(\bb)=\left\{0^4,(1/2)^2,(1/4)^1,(3/4)^1\right\}$.

The triangulation procedure of section \ref{secTriang} consists of five
 triangulation steps. In the intermediate steps, the vector $\vec{\nu}$ of $w$-values takes the following values (after reordering):
 $$
 \begin{array}{l}
 \vec{\nu}=[9/2,\,11/4,\,9/4,\,2,\,3/2,\,1,\,0,\,0],\\
 \vec{\nu}=[9/2,\,11/4,\,9/4,\,3/2,\,1,\,1,\,0,\,0],\\
 \vec{\nu}=[9/2,\,11/4,\,9/4,\,1,\,1,\,1/2,\,0,\,0].
 \end{array}
$$
The final upper triangular matrix is:
$$
T_{\op{triang}}=
\begin{pmatrix}
1&-1&-11&12&16&4&-4&0\\
0&1&-3&1&0&-2&0&0\\
0&0&1&-3&1&0&-2&0\\
0&0&0&1&1&1&0&0\\
0&0&0&0&1&1&1&0\\
0&0&0&0&0&1&0&0\\
0&0&0&0&0&0&1&0\\
0&0&0&0&0&0&0&1\\
\end{pmatrix},
\qquad
\vec{\nu}=\begin{pmatrix}
9/2\\
11/4\\
9/4\\
1\\
1/2\\
0\\
0\\
0
\end{pmatrix}.
$$
The vector $\vec{\nu}$ describes the canonical maximal $w$-values:
$$
\delta_0=\delta_1=\delta_2=0,\ \delta_3=1/2,\ \delta_4=1,\ \delta_5=9/4,\ \delta_6=11/4,\ \delta_7=9/2.
$$

The RNF routine of section \ref{secRNF} leads to:
$$
T_{\op{RNF}}=
\begin{pmatrix}
1&-1&-3&4&8&-12&12&0\\
0&1&1&1&0&2&0&0\\
0&0&1&1&1&0&2&0\\
0&0&0&1&1&1&0&0\\
0&0&0&0&1&1&1&0\\
0&0&0&0&0&1&0&0\\
0&0&0&0&0&0&1&0\\
0&0&0&0&0&0&0&1\\
\end{pmatrix},
\qquad
\vec{\nu}=\begin{pmatrix}
9/2\\
11/4\\
9/4\\
1\\
1/2\\
0\\
0\\
0
\end{pmatrix}.
$$

On the other hand, the basis in Hermite Normal Form would be:
$$
T_{\op{HNF}}=
\begin{pmatrix}
1&3&1&0&0&4&12&0\\
0&1&0&0&3&2&2&0\\
0&0&1&1&1&0&2&0\\
0&0&0&1&1&1&0&0\\
0&0&0&0&1&0&0&0\\
0&0&0&0&0&1&0&0\\
0&0&0&0&0&0&1&0\\
0&0&0&0&0&0&0&1\\
\end{pmatrix},
\qquad
\vec{\nu}=\begin{pmatrix}
4\\
2\\
9/4\\
1\\
0\\
0\\
0\\
0
\end{pmatrix}.
$$
This corresponds to a simpler basis indeed. However, since the $w$-values are not the maximal ones, this basis is not reduced, by Theorem \ref{reducedtriang}.

Finally, let us illustrate the computation of $T_{\op{RNF}}$ by the MaxMin algorithm.
For $\alpha\in L$, let us denote:
$$
\vec{w}(\alpha)=(w_1(\alpha),w_2(\alpha),w_3(\alpha)).
$$
Along the execution of the Montes algorithm, we compute and store the $w$-vectors of all $\phi$-polynomials and Okutsu approximations:\medskip

\as{1.2}
\begin{tabular}{|c|c|c|c|c|c|}
\hline
$\phi$&$x$&$x^2+2x+2$&$\phi_{\p_1}$&$\phi_{\p_2}$&$\phi_{\p_3}$\\
\hline
$\vec{w}(\phi(\t))$&$(1/2,5/2,0)$&$(7/4,1,0)$&$(\infty,2,0)$&$(1,\infty,0)$&$(0,0,\infty)$\\
\hline
\end{tabular}\medskip
\as{1}

The coordinates with value $\infty$ just indicate that the values $w_i(\phi_{\p_i})$ for $i=1,2,3$ can become arbitrarily large for a proper improvement of the Okutsu approximations with the SFL algorithm, while the values $w_i(\phi_{\p_j})$ remain constant for $i\ne j$.

With this information at hand, the MaxMin algorithm constructs monic polynomials $g_0,\dots,g_7$ of degree $0,\dots,7$ attaining the maximal $w$-values. By \cite[Thm. 2.6]{hayden} these polynomials may be obtained as adequate products of $\phi$-polynomials and Okutsu approximations. After a very simple search \cite[Thm. 3.3]{hayden}, we take:
\begin{equation}\label{gs}
\begin{array}{lll}
g_0=1,& g_3=x\,\phi_{\p_3},& g_6=x^2(x^2+2x+2)\,\phi_{\p_3}, \\
g_1=x,& g_4=(x^2+2x+2)\,\phi_{\p_3},& g_7=x\,\phi_{\p_1}\phi_{\p_3}.\\
g_2=x^2,&g_5=x(x^2+2x+2)\,\phi_{\p_3},&
\end{array}
\end{equation}
giving rise directly to the sequence of canonical $w$-values:
$$
\delta_0=\delta_1=\delta_2=0,\ \delta_3=1/2,\ \delta_4=1,\ \delta_5=9/4,\ \delta_6=11/4,\ \delta_7=9/2.
$$

In order to have $w(g_j(\t))=\delta_j$ for all $0\le j<8$, the conditions $w_1(\phi_{\p_1}(\t))=\infty$,  $w_3(\phi_{\p_3}(\t))=\infty$ may be replaced by:
$$
w_1(\phi_{\p_1}(\t))\ge 4,\quad w_3(\phi_{\p_3}(\t))\ge 9/2.
$$
For the  concrete choices for the Okutsu approximations provided by the Montes algorithm we have $w_1(\phi_{\p_1}(\t))=15/4$ and $w_3(\phi_{\p_3}(\t))=1$, which is not enough for our purposes. A single iteration of the SFL algorithm for each factor yields the right improvements: 
$$
\begin{array}{ll}
\phi_{\p_1}=x^4+32x^3+52x^2+48x+28,& w_1(\phi_{\p_1})=9/2\ge 4,\\
\phi_{\p_3}=x^2-x+1,& w_3(\phi_{\p_3})=8\ge 9/2.
\end{array}
$$
Now, we may execute the computation (\ref{gs}) of $g_0,\dots,g_7$ with these concrete values of $\phi_{\p_1}$, $\phi_{\p_3}$ provided by the SFL algorithm. 
In this way, we obtain a triangular matrix:
$$
T_{\op{MaxMin}}=
\begin{pmatrix}
1&31&21&28&32&20&28&0\\
0&1&1&1&0&2&0&0\\
0&0&1&1&1&0&2&0\\
0&0&0&1&1&1&0&2\\
0&0&0&0&1&-1&1&0\\
0&0&0&0&0&1&0&0\\
0&0&0&0&0&0&1&0\\
0&0&0&0&0&0&0&1\\
\end{pmatrix},
\qquad
\vec{\nu}=\begin{pmatrix}
9/2\\
11/4\\
9/4\\
1\\
1/2\\
0\\
0\\
0
\end{pmatrix},
$$
which yields the canonical matrix $T_{\op{RNF}}$ after applying the RNF routine of section \ref{secRNF}.

\end{document}